\documentclass[a4paper,11pt]{article}

\usepackage{amsmath}
\usepackage{amsthm}
\usepackage{amssymb}
\usepackage[dvips]{graphicx}

\usepackage{epsfig}
\usepackage{psfrag}
\usepackage{subfigure}
\usepackage{graphicx}

\usepackage[mathscr,mathcal]{eucal}
\usepackage{enumerate}

\usepackage{t1enc}
\usepackage[latin2]{inputenc}

\def \Z {\mathbb Z}

\def\cC{\mathcal{C}}

\def\cF{\mathcal{F}}

\def\cH{\mathcal{H}}

\def\cL{\mathcal{L}}

\newcommand{\expect}[1]{\ensuremath{\mathbf{E}\left(#1\right)}}

\newcommand{\condexpect}[2]{\ensuremath{\mathbf{E}\big(#1\bigm|#2\big)}}

\def\one{1\!\!1}

\renewcommand{\d}{\,\mathrm d}

\newcommand{\norm}[1]{\left\|\,{#1}\,\right\|}

\newtheorem {lemma}{Lemma}

\newtheorem {proposition}{Proposition}

\newtheorem* {theorem*}{Theorem}
\newtheorem* {thm*}{Theorem}
\newtheorem* {lemma*}{Lemma}
\newtheorem* {lem*}{Lemma}
\newtheorem* {corollary*}{Corollary}
\newtheorem* {cor*}{Corollary}
\newtheorem* {proposition*}{Proposition}
\newtheorem* {prop*}{Proposition}
\newtheorem* {definition*}{Definition}
\newtheorem* {def*}{Definition}
\newtheorem* {conjecture*}{Conjecture}

\newtheorem* {theoremkv*} {Theorem KV}
\newtheorem* {theoremmw*} {Theorem MW}
\newtheorem* {theoremssc*} {Theorem SSC}
\newtheorem* {theoremgsc*} {Theorem GSC}

\theoremstyle{definition}
\newtheorem* {remark*}{Remark}
\newtheorem* {remarks}{Remarks}
\newtheorem* {rem*}{Remark}

\newtheorem* {assumption*}{Assumption}

\def\be{\begin{align}}
\def\ee{\end{align}}

\def\bea{\begin{eqnarray}}
\def\eea{\end{eqnarray}}

\title{Comment on a theorem of M. Maxwell and M. Woodroofe
\footnote{Work partially supported by OTKA (Hungarian National Research Fund) grant K100473.}}

\author{
B\'alint T\'oth
\thanks{Institute of Mathematics, TU Budapest. E-mail: {\tt balint@math.bme.hu}.}
\thanks{School of Mathematics, University of Bristol. E-mail: {\tt balint.toth@bristol.ac.uk}.}
}

\begin{document}

\maketitle

\begin{abstract}

We present a streamlined derivation of the theorem of M. Maxwell and M. Woodroofe,\cite{maxwell_woodroofe_00}, on martingale approximation of additive functionals of stationary Markov processes, from the non-reversible version of the Kipnis-Varadhan theorem.

\end{abstract}

\section{Setup}
\label{s:setup}

Let $(\Omega, \cF, \pi)$ be a probability space: the state space of a \emph{stationary and ergodic} Markov process $t\mapsto\eta(t)$. We put ourselves in the real Hilbert space $\cH:=\cL^2(\Omega, \pi)$, with inner product $(\varphi, \psi):=\int_\Omega \varphi(\omega)\psi(\omega)\d\pi(\omega)$. Denote by $P_t$ the Markov semigroup of conditional expectations acting on $\cH$:
\[
P_t:\cH\to\cH,
\qquad
P_t \varphi (\omega):=\condexpect{\varphi(\eta_t)}{\eta_0=\omega},
\qquad t\ge0.
\]
This is assumed to be a strongly continuous contraction semigroup, whose \emph{infinitesimal generator} is denoted by $G$, which is a well-defined (possibly unbounded) closed linear operator of Hille-Yosida type on $\cH$. It is assumed that there exists a dense core $\cC\subseteq\cH$ on which $G$ is decomposed as
\begin{equation*}
\label{SandA1}
G=-S+A,
\end{equation*}
where $S$ is Hermitian and positive semidefinite, while $A$ is skew-Hermitian:
\begin{equation*}
\label{SandA2}
\forall \varphi,\psi\in\cC:
\qquad
(\varphi, S \psi)=(S \varphi,\psi),
\quad
(\varphi, S \varphi)\ge0,
\quad
(\varphi, A \psi)=-(A \varphi,\psi).
\end{equation*}
Finally, it is assumed that $S$, respectively, $A$ are essentially self-adjoint, respectively, essentially skew-self-adjoint on the core $\cC$. The operator $S^{1/2}$ appearing in the forthcoming arguments is defined in terms of the spectral theorem.

Let $f\in\cH$, be such that $(f, \one) = \int_\Omega f \d\pi=0$, where $\one\in\cL^2(\Omega,\pi)$ is the constant function $\one(\omega)\equiv1$. We ask about CLT/invariance principle, as $N\to\infty$, for
\[
N^{-1/2}\int_0^{Nt} f(\eta(s))\d s.
\]

We denote:
\begin{align*}
R_\lambda
&
:=
\int_0^\infty e^{-\lambda s} P_s \d s
=
\big(\lambda I-G\big)^{-1},
&&
u_\lambda:=R_\lambda f,
&&
\lambda>0,
\\[8pt]
V_t
&
:=
\int_0^t P_s \d s = G^{-1}(I-P_t),
&&
v_t:=V_t f,
&&
t>0.
\end{align*}

Recall the non-reversible version of the Kipnis-Varadhan theorem and the theorem of Maxwell and Woodroofe about the CLT problem mentioned above:

\begin{theoremkv*}
\label{thm:kv}
With the notation and assumptions as before, if the following two limits hold in $\cH$ (in norm topology):
\begin{align}
\label{conditionKV1}
&
\lim_{\lambda\to0}
\lambda^{1/2} u_\lambda=0,
\\[8pt]
\label{conditionKV2}
&
\lim_{\lambda\to0} S^{1/2} u_\lambda=:w\in\cH,
\end{align}
then
\[
\sigma^2:=2\lim_{\lambda\to0}(u_\lambda,f)=2\norm{w}^2\in[0,\infty),
\]
exists, and there also exists a zero mean, $\cL^2$-martingale $M(t)$ adapted to the filtration of the Markov process $\eta(t)$, with stationary and ergodic increments, and variance
\[
\expect{M(t)^2}=\sigma^2t,
\]
such that
\[
\lim_{N\to\infty} N^{-1} \expect{\big(\int_0^N
f(\eta(s)) \d s-M(N)\big)^2} =0.
\]
In particular, if $\sigma>0$, then the finite dimensional marginal distributions of the rescaled process $t\mapsto \sigma^{-1} N^{-1/2}\int_0^{Nt}f(\eta(s))\d s$ converge to those of a standard $1d$ Brownian motion.
\end{theoremkv*}

Conditions \eqref{conditionKV1} and \eqref{conditionKV2} of Theorem KV are jointly equivalent to the following
\begin{align}
\label{conditionBT}
\lim_{\lambda,\lambda'\to0}(\lambda+\lambda')(u_\lambda,u_{\lambda'})=0.
\end{align}
Indeed, straightforward computations yield:
\begin{align}
\label{KV1+KV2=BT}
(\lambda+\lambda')(u_\lambda,u_{\lambda'}) =
\norm{S^{1/2}(u_\lambda-u_{\lambda'})}^2 + \lambda \norm{u_\lambda}^2 +
\lambda' \norm{u_{\lambda'}}^2.
\end{align}

\begin{theoremmw*}
\label{thm:mw}
With the notation and assumptions as before, if:
\begin{align}
\label{conditionMW}
\int_0^\infty t^{-3/2}\norm{v_t} \d t
<\infty,
\end{align}
then the martingale approximation and CLT from Theorem KV hold.
\end{theoremmw*}

\begin{remarks}

\begin{enumerate}[$\circ$]

\item
The reversible version (when $A=0$) of Theorem KV appears in the celebrated paper \cite{kipnis_varadhan_86}. In that case conditions \eqref{conditionKV1} and \eqref{conditionKV2} are equivalent and  the proof relies on spectral calculus. The non-reversible formulation of Theorem KV appears -- in discrete-time Markov chain, rather than continuous-time Markov process setup and with condition \eqref{conditionBT} -- in \cite{toth_86}. Its proof follows the original proof from \cite{kipnis_varadhan_86}, with spectral calculus methods replaced by resolvent calculus.

\item
Theorem MW appears in \cite{maxwell_woodroofe_00}. Its proof contains elements in common with the arguments of the proof of Theorem KV. However, in the original formulation it's not transparent that Theorem MW is actually a direct consequence of  Theorem KV.

\item
For full historical record of the circle of ideas and results related to Theorem KV (as, e.g., the various sector conditions) and a wide range of applications to tagged particle diffusion in interacting particle systems, random walks and diffusions in random environment, other random walks and diffusions with long memory, etc., see the recent monograph \cite{komorowski_landim_olla_12}.

\end{enumerate}

\end{remarks}

\section{Theorem MW from Theorem KV}
\label{s:MW_from_KV}

\begin{proposition}
\label{prop}
If there exists a decreasing sequence $\lambda_k\searrow 0$ such that
\begin{align}
\label{sumcondition}
\sum_{k=1}^\infty
\sqrt{\lambda_{k-1}}\norm{u_{\lambda_k}}<\infty,
\end{align}
then conditions \eqref{conditionKV1} and \eqref{conditionKV2} of Theorem KV hold.
\end{proposition}

\begin{remark*}

\begin{enumerate}[$\circ$]

\item
Proposition \ref{prop} also sheds some light on the conditions of Theorem KV: It shows that \eqref{conditionKV1} alone is just marginally short of being sufficient.

\end{enumerate}

\end{remark*}

\begin{proof}[Proof of Proposition \ref{prop}]

Note first that from \eqref{KV1+KV2=BT}, by Schwarz's inequality it follows that
\begin{align}
\label{auxbound}
2\norm{S^{1/2}(u_\lambda-u_{\lambda'})}^2
\le
(\lambda-\lambda')
(\norm{u_{\lambda'}}^2 - \norm{u_\lambda}^2)
\le
\lambda\norm{u_{\lambda'}}^2
+
\lambda'\norm{u_\lambda}^2.
\end{align}
Hence, $\lambda\mapsto \norm{u_\lambda}$ is monotone decreasing and
\begin{equation}
\label{inbetween1}
\max_{\lambda_k\le \lambda\le \lambda_{k-1}} \sqrt{\lambda}\norm{u_\lambda} \le \sqrt{\lambda_{k-1}} \norm{u_{\lambda_k}}.
\end{equation}
The summability condition \eqref{sumcondition} and the bound \eqref{inbetween1} clearly imply \eqref{conditionKV1}.

From \eqref{auxbound} we also get
\[
\label{crossbound_to_seq}
\norm{ S^{1/2} (u_{\lambda_{k}} - u_{\lambda_{k-1}}) }
\le
\sqrt{\lambda_{k-1}} \norm{u_{\lambda_k}}.
\]
Hence, by the assumption \eqref{sumcondition}
\[
\sum_{k=1}^\infty \norm{S^{1/2}(u_{\lambda_{k}}-u_{\lambda_{k-1}})}<\infty,
\]
and thus
\begin{align}
\label{limit_for_seq}
\lim_{k\to\infty}S^{1/2}u_{\lambda_{k}} =: w\in\cH
\end{align}
exists. Now, using again \eqref{auxbound} we have
\begin{equation}
\label{inbetween2}
\lim_{k\to\infty}
\max_{\lambda_k\le \lambda\le \lambda_{k-1}}
\norm{S^{1/2}(u_{\lambda_{k}}-u_{\lambda})}
\le
\lim_{k\to\infty}
\sqrt{\lambda_{k-1}}\norm{u_{\lambda_k}}
=0.
\end{equation}
Finally, \eqref{limit_for_seq} and \eqref{inbetween2} jointly yield \eqref{conditionKV2}.

\end{proof}

The following is essentially Lemma 1 from \cite{maxwell_woodroofe_00}. We reproduce it only for sake of completeness.

\begin{lemma}
\label{lem:mw}
Condition \eqref{conditionMW} of Theorem MW implies the summability condition \eqref{sumcondition} of Proposition \ref{prop}, with any exponential sequence $\lambda_k=\delta^k$, $\delta\in(0,1)$.
\end{lemma}

\begin{proof}[Proof of Lemma \ref{lem:mw}]

This is straightforward computation. Note first that
\[
u_\lambda=\lambda\int_0^\infty e^{-\lambda t} v_t \d t,
\qquad
\norm{u_\lambda}
\le
\lambda \int_0^\infty e^{-\lambda t} \norm{v_t} \d t.
\]
Thus,
\begin{align}
\label{what}
\sum_{k=0}^\infty \delta^{k/2} \norm{u_{\delta^{k}}}
&
\le
\int_0^\infty \left(\sum_{k=0}^\infty (t \delta^k)^{3/2} e^{-t \delta^k}\right) t^{-3/2}\norm{v_t} \d t.
\end{align}
Next we prove that for any $\delta\in(0,1)$
\begin{align}
\label{gamma}
\sup_{0\le t <\infty}
\sum_{k=-\infty}^\infty (t \delta^k)^{3/2} e^{-t \delta^k}
\le
\left(\frac{3}{2e}\right)^{3/2} + \frac{\sqrt{\pi}}{2(1-\delta)}.
\end{align}
From \eqref{what} and \eqref{gamma} the statement of the lemma follows.

Fix $t\in[0,\infty)$, $\delta\in(0,1)$ and denote $u_k:= t \delta^k$. Since the function $[0,\infty)\ni u\mapsto u^{1/2}e^{-u}$ is \emph{strictly unimodular}, there exists a \emph{unique} $k^*=k^*(t,\delta)\in\Z$ such that
\[
u_k^{1/2} e^{-u_k}
=
\begin{cases}
\displaystyle
\min_{u_{k+1}\le u\le u_{k}} u^{1/2}e^{-u}
& \text{ if } k<k^*,
\\
\displaystyle
\min_{u_{k}\le u\le u_{k-1}} u^{1/2}e^{-u}
& \text{ if } k>k^*.
\end{cases}
\]
Then the sum on the left hand side of  \eqref{gamma} is:
\begin{align*}
&
\sum_{k=-\infty}^\infty u_k^{3/2} e^{-u_k}
=
\\
&\hskip1cm
\frac{1}{1-\delta}
\sum_{k=-\infty}^{k^*-1} (u_k-u_{k+1})u_k^{1/2} e^{-u_k}
+u_{k^*}^{3/2} e^{-u_{k^*}}+
\frac{\delta}{1-\delta}
\sum_{k=k^*+1}^{\infty} (u_{k-1}-u_k)u_k^{1/2} e^{-u_k}
\le
\\
&\hskip1cm
\sup_{0\le u \le\infty} u^{3/2} e^{-u} + \frac{1}{1-\delta}
\int_0^{\infty} u^{1/2}e^{-u} \d u.
\end{align*}
Hence \eqref{gamma}, and the statement of the lemma follows.

\end{proof}

\end{document}